\documentclass[12pt,reqno]{amsart}
\usepackage{amsmath}
\usepackage{amssymb}
\usepackage{amstext}
\usepackage{a4wide}
\usepackage{txfonts}
\usepackage{graphicx} 
\usepackage{bm}
\usepackage[numbers,sort&compress]{natbib}
\allowdisplaybreaks \numberwithin{equation}{section}
\usepackage{color}
\usepackage{cases}

\numberwithin{equation}{section}

\newtheorem{theorem}{Theorem}[section]
\newtheorem{proposition}[theorem]{Proposition}
\newtheorem{corollary}[theorem]{Corollary}
\newtheorem{lemma}[theorem]{Lemma}

\newtheorem*{Yudovich's Theorem}{Yudovich's Theorem}

\theoremstyle{definition}

\newtheorem{definition}[theorem]{Definition}

\theoremstyle{remark}
\newtheorem{remark}[theorem]{Remark}

\begin{document}

\title
[Nonlinear stability of plane ideal flows in a periodic channel]{Nonlinear stability of plane ideal flows in a periodic channel}

 \author{Guodong Wang}

\address{School of Mathematical Sciences, Dalian University of Technology, Dalian 116024, PR China}
\email{gdw@dlut.edu.cn}


\begin{abstract}
In this paper, we establish two stability theorems for steady or traveling solutions of the two-dimensional incompressible Euler equation in a finite periodic channel, extending Arnold's classical work from the 1960s. Compared to Arnold's approach, we employ a compactness argument rather than relying on the definiteness of the second variation of the energy-Casimir functional. The isovortical property of the Euler equation and Burton's rearrangement theory play an essential role in our analysis.
As a corollary, we prove for the first time the existence of a class of stable non-shear flows when the ratio of the channel's height to its length is less than or equal to $\sqrt{3}/2.$ Two rigidity results are also obtained as byproducts. 
\end{abstract}

\maketitle
\tableofcontents
\section{Introduction and main results}\label{sec1}
\subsection{Euler equation in a periodic channel}
 Consider the Euler equation of an incompressible inviscid fluid of unit density in a finite periodic channel $D:=\mathbb T_L\times(0,H)$:
 \begin{equation}\label{euler}
\begin{cases}
\partial_t\mathbf v+(\mathbf v\cdot\nabla)\mathbf v=-\nabla P, &t\in\mathbb R,\,\mathbf x=(x_1,x_2)\in D,\\
\nabla\cdot\mathbf v=0,\\
\end{cases}
\end{equation}
where $L,H>0,$ $\mathbf v(t,\mathbf x)=(v_1(t,\mathbf x),v_2(t,\mathbf x))$ is the velocity field, and $P(t,\mathbf x)$ is the scalar pressure.
The following impermeability conditions are imposed on $\partial D=\{x_2=0\}\cup\{x_2=H\}$:
 \begin{equation}\label{bdc}
 v_2|_{x_2=0}=0,\quad  v_2|_{x_2=H}=0.
 \end{equation}
The divergence-free condition $\nabla\cdot\mathbf v=0$ and the impermeability conditions \eqref{bdc} ensure the existence of a stream function $\psi$ such that  
 \[
\mathbf v= \nabla^\perp \psi,  \quad\psi \,\mbox{ is piecewise constant on}\,\partial D,
\]
where $\nabla^\perp=(\partial_{x_2},-\partial_{x_1})$.
Moreover, such $\psi$ is unique up to an arbitrary constant.  
Denote by $\omega$ the scalar vorticity and by $Q$ the flux of the fluid in the $x_1$-direction, i.e., 
\[\omega:=\partial_{x_1} v_2-\partial_{x_2}v_1=-\Delta\psi,\quad Q:=\int_0^Hv_1 dx_2=\psi\big|^{x_2=H}_{x_2=0.}\]
Note that $Q$ does not depend on  $t$ or $x_1$.
Up to an arbitrary constant, the stream function $\psi$ can be expressed in terms of $\omega$ and  $Q$ as follows:
 \begin{equation}\label{psioq}
 \psi=\mathcal G\omega+\frac{Q}{H}x_2,
 \end{equation}
where $\mathcal G$ is the inverse of $-\Delta$ subject to zero Dirichlet boundary conditions.
From \eqref{psioq}, we see that the fluid motion in $D$ can be described from three equivalent perspectives: the velocity $\mathbf v,$ the stream function $\psi$, and the vorticity-flux pair $(\omega,Q)$.  In terms of the vorticity-flux pair $(\omega,Q)$, the Euler equation \eqref{euler} can be written as
\begin{equation}\label{vfeq}
\partial_t\omega+\left(\nabla^\perp \mathcal G\omega+\frac{Q }{H}\mathbf e_1\right)\cdot\nabla\omega=0,
\end{equation}
where $\mathbf e_1:=(1,0)$. In terms of the stream function, the Euler equation \eqref{euler} can be written as
\begin{equation}\label{sfeq}
\partial_t(\Delta\psi)+ \nabla^\perp \psi\cdot\nabla(\Delta\psi)=0.
\end{equation}
Global well-posedness of the two-dimensional incompressible Euler equation in various function spaces is a classical result and can be found in many classical papers or textbooks, such as \cite{BV, De, DM, MB, MP, Y}. In particular, given smooth initial data, there is a smooth global solution.

Based on \eqref{vfeq} and \eqref{sfeq},  we can specify the effect of the flux $Q$ on the fluid motion by a straightforward computation.
\begin{lemma}\label{lm001}
For any $\lambda\in\mathbb R$, 
\begin{itemize}
\item[(i)] $(\omega(t,x_1,x_2),{Q})$ solves \eqref{vfeq} if and only if $\left(\omega\left(t,x_1 -\lambda t,x_2\right),{Q}+\lambda H\right)$ solves \eqref{vfeq}, and
\item[(ii)]$\psi(t,x_1,x_2)$ solves \eqref{sfeq} if and only if $\psi\left(t,x_1 -\lambda t,x_2\right)+\lambda x_2$ solves \eqref{sfeq}.
\end{itemize}
\end{lemma}

Obviously, any $\bar\psi\in C^2(\bar D)$ that depends only on $x_2$ is a steady solution to \eqref{sfeq}. Such solutions are called shear solutions. 
Another important class of solutions is associated with a certain semilinear elliptic equation. More specifically, according to Lemma \ref{lm001}, 
if $\bar\psi\in C^2(\bar D)$ satisfies
\begin{equation}\label{tvs1}
\begin{cases}
-\Delta \bar\psi =g(\bar\psi -\lambda x_2),\quad\mathbf x\in D,\\
 \bar\psi\,\mbox{ is piecewise constant on}\,\partial D,\\
\end{cases}
\end{equation}
for some $g\in C^1(\mathbb R)$ and some $\lambda\in\mathbb R$, then $\psi(t,x_1,x_2):= \bar\psi(x_1-\lambda t,x_2)$ solves \eqref{sfeq}. This is a steady or traveling solution, and is steady if $\lambda=0$ or $\bar\psi$ depends only on $x_2.$

The following three conservation laws play an important role in this paper, and their proofs can be found in many standard textbooks, such as \cite{BV,MB,MP}.
\begin{itemize}
  \item [(C1)] The kinetic energy $\mathcal E$ of the fluid is conserved,  
      \[ \mathcal E :=\frac{1}{2}\int_D|\mathbf v|^2d\mathbf x= \frac{1}{2}\int_D|\nabla \psi|^2d\mathbf x.\]
      By integration by parts,  $\mathcal E$  can be expressed in terms of $(\omega,Q)$ as follows:
\[\mathcal E=\frac{1}{2}\int_D\omega\mathcal G\omega  d\mathbf x+  \frac{LQ^2}{2H}.\]
Define
\begin{equation}\label{defoe}
E(\omega):=\frac{1}{2}\int_D\omega\mathcal G\omega  d\mathbf x.
\end{equation}
Then $E$ is also conserved.
  \item [(C2)]The impulse $I$ (linear momentum parallel to the $x_1$-axis) is conserved,  
      \begin{equation}\label{defoi}
      I(\omega):=\int_Dx_2\omega d\mathbf x.
      \end{equation}
     For $v\in L^1(D)$, define
     \[\mathcal I_v:=\{w\in L^1(D)\mid I(w)=I(v)\}.\]
     Then the impulse conservation can be written as 
     \[\omega(t,\cdot)\in\mathcal I_{\omega(0,\cdot)}\quad\forall\,t\in\mathbb R.\] 
  \item [(C3)]The isovortical property holds, i.e., \[\omega(t,\cdot)\in\mathcal R_{\omega(0,\cdot)}\quad\forall\,t\in\mathbb R,\] 
      where $\mathcal R_v$ denotes the rearrangement class of any Lebesgue measurable function $v:D\to\mathbb R,$ i.e.,
      \begin{equation}\label{deforc}
      \mathcal R_{v}:=\left\{w:D\to\mathbb R\mid |\{ w >a\}|=|\{ v >a\}|\,\,\,\forall\,a\in\mathbb R\right\}.
      \end{equation} 
     Here and hereafter, we denote by $|S|$ the Lebesgue measure of any Lebesgue measurable set $S\subset\mathbb R^2.$   
 \end{itemize}
Note that as a consequence of (C3), there exist infinitely many integral invariants of the form
      $\int_DF(\omega)d\mathbf x,$
      where $F:\mathbb R\to\mathbb R$ is Borel measurable.
      These integral invariants are called Casimirs.

\subsection{Arnold's stability theorems}
For the reader's convenience, we provide a brief review of Arnold's seminal paper \cite{A1}, in which he proposed the energy-Casimir method and established two types of sufficient conditions for the stability of two-dimensional steady Euler flows.
We note that throughout this paper, ``stability" always refers to Lyapunov stability under the nonlinear dynamics governed by \eqref{euler}.

Given a steady flow with vorticity $\bar\omega,$ Arnold computed the first and the second variations of the following energy-Casimir functional $EC$ at $\bar\omega:$
\[EC(\omega)=E(\omega)-\int_DF(\omega)d\mathbf x,\]
where $F\in C^2(\mathbb R)$. Here, for simplicity, we only consider flows with zero flux (i.e., $Q=0$); in this case, $E(\omega)$ is exactly the kinetic energy.
It was shown that
\[ \delta EC|_{\bar\omega}(v)=\int_D v\mathcal G\bar\omega d\mathbf x-\int_Df(\bar\omega) vd\mathbf x,\quad f:=F',\]
\[ \delta^2 EC|_{\bar\omega}(v)=\int_Dv\mathcal Gvd\mathbf x-\int_Df'(\bar\omega) v^2d\mathbf x.\]
To ensure that the first variation  at $\bar\omega$ vanishes,  it suffices to require
 \begin{equation}\label{1stv}
\mathcal G\bar\omega=f(\bar\omega).
\end{equation}
As to the second variation, it is clear that 
if 
\begin{equation}\label{leq0}
f'(\bar\omega)< 0\,\,\,\mbox{ in}\,\,\,D,
\end{equation} then  $\delta^2 EC|_{\bar\omega}$ is positive definite in $L^2(D)$; and if  
\begin{equation}\label{geq0}
 f'(\bar\omega)>c_{ar}\,\,\,\mbox{ in}\,\,\,D,\quad c_{ar}:=\sup\left\{\int_D v\mathcal Gvd\mathbf x\mid \|v\|_{L^2(D)}=1 \right\},
\end{equation}
then  $\delta^2 EC|_{\bar\omega}$ is negative definite in $L^2(D)$.
Based on the above observation and using the fact that 
$EC$ is a conserved functional, Arnold proved the stability of the steady flow in $L^2(D)$ (the vorticity space) under the assumptions \eqref{1stv} and \eqref{leq0} (the first stability theorem), or under the assumptions \eqref{1stv} and \eqref{geq0} (the second stability theorem). The detailed statements and proofs can also be found in \cite{A2,A3,MP}.
In our setting of a finite periodic channel, it can be proved that \[c_{ar}=\frac{1}{\lambda_1}=\frac{H^2}{\pi^2},\]
 where $\lambda_1$ is the first eigenvalue of $-\Delta$ subject to zero Dirichlet boundary conditions.  As an example, the steady flow with vorticity $\bar\omega(x_1,x_2)=\sin x_2$ is stable if $H<\pi$.

Arnold's work has been discussed and extended by many authors from different perspectives; see \cite{AHMR,CG,GS,MS,Taylor,WS,WG1,WG2}. However, most of these works focus on the definiteness condition, which significantly restricts the applicability of the method, especially when the domain possesses certain symmetry. In fact, if a steady flow satisfies the definiteness condition, it must be the \emph{unique} minimizer or maximizer of some conserved functional (such as the 
$EC$ functional). On the other hand, the conserved functional is necessarily invariant under the action of the domain's symmetry group, which in turn forces the steady flow to be invariant under the same symmetry group.
 For example,  for any steady flow satisfying the conditions of  Arnold's stability theorems:
\begin{itemize} 
\item [(i)] If the domain is a disk, then the flow must be circular;
    \item [(ii)] If the domain is a flat two-torus or a two-sphere, then the flow must be trivial; 
    \item [(iii)] If the domain is a finite periodic channel, then the flow must be a shear flow.
 \end{itemize}
  See also \cite[Proposition 1.1]{CDG} for a more comprehensive discussion.

 To overcome the limitations caused by the definiteness condition,  
  Wirosoetisno and Shepherd \cite{WS99}  proposed using other Casimirs to bound the variation of the perturbed solutions. As a result, they proved the orbital stability of a class of sinusoidal flows on a flat 2-torus. Recently, Wirosoetisno and Shepherd's method was used by Constantin and Germain \cite{CG} to prove the stability of degree-2 zonal Rossby-Haurwitz waves on a rotating sphere.
However,  Wirosoetisno and Shepherd's method involves a large amount of complex calculations, making it inconvenient to apply.

Another effective approach that can avoid the definiteness condition is the compactness argument, which is more flexible and has a strong variational flavor. This approach was first used by Burton \cite{BAR} to study the stability of steady flows in a bounded, simply connected domain. More specifically, through a compactness argument in combination with his rearrangement theory, Burton \cite{BAR}  established a general stability criterion, stating that  \emph{any steady flow corresponding to an isolated local minimizer/maximizer of the kinetic energy (as a functional of the vorticity) relative to the rearrangement class of an $L^p$ function must be stable,  where $1<p<\infty.$} The term ``local minimizer/maximizer" is interpreted in the sense of the $L^p$ norm. It is easy to check that Arnold's stability theorems are, in fact, two special cases of Burton's stability criterion: \emph{the vorticity of the steady flow in Arnold's first (second) stability theorem corresponds to the unique minimizer (maximizer) of the kinetic energy relative to its rearrangement class.} Other steady flows that satisfy the conditions  of Burton's stability criterion include those studied in \cite{BG,CWCV,CWJMFM,LZCMP,WJDE,WTAMS,WMA,WG1,WG2}. 

In recent years, the authors of \cite{WZCV,CWZ,Wdisk} have proved the (orbital) stability of certain explicit flows in domains with good symmetry by using a compactness argument and studying the equimeasurable partition of a certain finite-dimensional function space.
The flows considered in \cite{WZCV,CWZ,Wdisk} neither obey the symmetry of the domain nor satisfy the isolatedness condition in Burton's stability criterion. In a broad sense, these results can be regarded as the "critical" case of Arnold's second stability theorem. The motivation for this paper is to extend the ideas in \cite{WZCV,CWZ,Wdisk} to investigate the stability of steady flows in a finite periodic channel, with particular emphasis on the existence of stable \emph{non-shear} flows.

It should be noted that the extension of Arnold's \emph{first} stability theorem is relatively straightforward. In fact, for any $v\in L^p(D),$ $1<p<\infty,$ the following three statements are equivalent (cf. \cite[Theorem 2.1]{BMcL}):
 \begin{itemize}
  \item[(i)] $v$ is a local minimizer of $E$ relative to $\mathcal R_{v}$;
 \item [(ii)]  $v$ is the unique minimizer of $E$ relative to $\mathcal R_{v}$;
 \item[(iii)] there exists a decreasing function $\phi:\mathbb R\to\mathbb R\cup\{\pm\infty\}$\footnote{That is, $\phi(s_1)\leq \phi(s_2)$ whenever $s_1\leq s_2$.} such that $ v=\phi(\mathcal G v)$ a.e. in $D$.
 \end{itemize}
If the vorticity of a steady flow satisfies any of the above three conditions, then stability follows from Burton's stability criterion in \cite{BAR}.

 \subsection{Main results}
 
Before presenting the main results, we introduce the following Laplace eigenvalue problem:
 \begin{equation}\label{lep1}
\begin{cases}
-\Delta u=\Lambda u,\quad\quad\mathbf x\in D,\\
 u|_{x_2=0}=u|_{x_2=H}={\rm constant},\\
 \int_Dud\mathbf x=0.
 \end{cases}
\end{equation}
Denote by $\Lambda_1$ the first eigenvalue, and by $\mathbb E_1$ the first eigenspace. According to Proposition \ref{propevp} in Section \ref{sec2}, $\Lambda_1$ and $\mathbb E_1$ are determined by ${H}/{L}$ as follows.
\begin{itemize}
\item[(1)] If ${H}/{L}>{\sqrt{3}}/{2}$, then
\begin{equation}\label{hl1}
\begin{cases}
\Lambda_1= \frac{4\pi^2}{H^2},\\
 \mathbb E_1={\rm span}\left\{\sin \left(\frac{2\pi x_2}{H}\right),\,\cos\left(\frac{2\pi x_2}{H}\right)\right\}.
\end{cases}
\end{equation}
\item[(2)] If ${H}/{L}<{\sqrt{3}}/{2}$, then 
\begin{equation}\label{hl2}
\begin{cases}
\Lambda_1=\frac{4\pi^2}{L^2} + \frac{\pi^2}{H^2},\\
 \mathbb E_1={\rm span}\left\{\cos\left(\frac{2\pi x_1}{L}\right)\sin\left(\frac{\pi x_2}{H}\right),\,
\sin\left(\frac{2\pi x_1}{L}\right)\sin\left(\frac{\pi x_2}{H}\right)\right\}.
\end{cases}
\end{equation}
\item[(3)] If  ${H}/{L}={\sqrt{3}}/{2}$, then 
\begin{equation}\label{hl3}
\begin{cases}
 \Lambda_1=\frac{4\pi^2}{H^2}=\frac{4\pi^2}{L^2} + \frac{\pi^2}{H^2},\\
 \mathbb E_1={\rm span}\left\{\sin \left(\frac{2\pi x_2}{H}\right),\,\cos\left(\frac{2\pi x_2}{H}\right),\,\cos\left(\frac{2\pi x_1}{L}\right)\sin\left(\frac{\pi x_2}{H}\right),\,
\sin\left(\frac{2\pi x_1}{L}\right)\sin\left(\frac{\pi x_2}{H}\right)\right\}.
\end{cases}
\end{equation}
\end{itemize}

  Our main results provide two sufficient conditions for the stability of  steady/traveling flows given by \eqref{tvs1}, which can be regarded as extensions of Arnold's second stability  theorem. Note that, according to Arnold's stability theorems, a flow given by \eqref{tvs1} is stable if $g'<0$ (the first stability theorem) or $0<g'<H^2/\pi^2$ (the second stability theorem). In this paper, we show that the condition  $0<g'<H^2/\pi^2$  in the second stability theorem can be relaxed to
$0\leq g'\leq \Lambda_1$.

 Our first result is stated as follows.
  \begin{theorem}\label{thm1}
 Suppose that $\bar\psi\in C^2(\bar D)$ satisfies \eqref{tvs1}, where $g\in C^1(\mathbb R)$ and $\lambda\in\mathbb R$.   Denote $\bar\omega:=-\Delta\bar\psi$.
   Fix $1<p<\infty$.
  If 
   \begin{equation}\label{gcond1}
   0\leq \min_{\bar D}g'(\bar\psi-\lambda x_2)\leq \max_{\bar D}g'(\bar\psi-\lambda x_2)<\Lambda_1,
   \end{equation}
      then  
      \begin{itemize}
      \item[(i)]$\bar\psi$ depends only on $x_2$, and 
      \item[(ii)] the associated steady flow is stable in the following sense:   for any $\varepsilon>0,$ there exists some $\delta>0,$ such that for any sufficiently smooth Euler flow with vorticity $\omega(t,\mathbf x)$,  
   \begin{equation}\label{stabless1}
   \|\omega(0,\cdot)-\bar\omega\|_{L^p(D)} <\delta \quad\Longrightarrow\quad \|\omega(t,\cdot)- \bar\omega  \|_{L^p(D)}<\varepsilon\quad \forall\,t\in\mathbb R.
  \end{equation}
 \end{itemize}
     \end{theorem}

Our second result is the following.
   \begin{theorem}\label{thm2}
 Suppose that $\bar\psi\in C^2(\bar D)$ satisfies \eqref{tvs1}, where $g\in C^1(\mathbb R)$ and $\lambda\in\mathbb R$.   Denote $\bar\omega:=-\Delta\bar\psi$.
   Fix $1<p<\infty$.
   If
     \begin{equation}\label{gcond2}
   0\leq \min_{\bar D}g'(\bar\psi-\lambda x_2)\leq \max_{\bar D}g'(\bar\psi-\lambda x_2)\leq \Lambda_1,
   \end{equation}
     then
        \begin{itemize}
      \item[(i)] $\bar\psi$ can be decomposed as
     \[\bar\psi=\bar\psi^e+\bar\psi^s,\]
where $\bar\psi^e\in\mathbb E_1$ and  $\bar\psi^s$ depends only on $x_2$, and
\item[(ii)] the associated steady/traveling flow is orbitally stable in the following sense:  for any $\varepsilon>0,$ there exists some $\delta>0,$ such that for any sufficiently smooth Euler flow with vorticity $\omega(t,\mathbf x)$, if 
    \[ \|\omega(0,\cdot)-\bar\omega\|_{L^p(D)} <\delta, \]
    then for any $t\in\mathbb R,$ there exists some $\alpha_t\in\mathbb R$ such that 
  \[ \|\omega(t,\cdot)- \bar\omega(\cdot+\alpha_t\mathbf e_1) \|_{L^p(D)}<\varepsilon.\]
 \end{itemize}
\end{theorem}

We give some remarks about Theorems \ref{thm1} and \ref{thm2}.

\begin{remark}
Without loss of generality, we assume that the stream functions mentioned in this paper vanish on $\{x_2=0\}$.
Then the stability conclusions in Theorems \ref{thm1} and \ref{thm2} can also be expressed in terms of $W^{2,p}$ norm of the stream function. For example, the stability conclusion in Theorem  \ref{thm1} can be stated as: \emph{for any $\varepsilon>0,$ there exists some $\delta>0,$ such that for any sufficiently smooth Euler flow with stream function $\psi(t,\mathbf x)$,  
\[
   \|\psi(0,\cdot)-\bar\psi\|_{W^{2,p}(D)} <\delta \quad\Longrightarrow\quad \|\psi(t,\cdot)-\bar\psi\|_{W^{2,p}(D)} <\varepsilon\quad \forall\,t\in\mathbb R.
\]}

\end{remark}

\begin{remark}
  Denote by $\mathcal O_{\bar\omega}$ the orbit of $\bar\omega$ under the action of the translation group in the $x_1$-direction, i.e.,
\begin{equation}\label{defoo}
\mathcal O_{\bar\omega}:=\left\{v:D\to \mathbb R\mid v(x_1,x_2)=\bar\omega(x_1+\alpha,x_2),\, \alpha\in\mathbb R\right\}.
\end{equation}
Then the conclusion of Theorem \ref{thm2}(ii) can also be stated as follows:  \emph{for any $\varepsilon>0,$ there exists some $\delta>0,$ such that for any sufficiently smooth Euler flow with vorticity $\omega(t,\mathbf x)$,  
      \begin{equation}\label{stabless2}
      \|\omega(0,\cdot)-\bar\omega\|_{L^p(D)} <\delta \quad\Longrightarrow\quad \min_{v\in\mathcal O_{\bar\omega}} \|\omega(t,\cdot)- v \|_{L^p(D)}<\varepsilon\quad\forall\,t\in\mathbb R.
      \end{equation}
      }
\end{remark}

\begin{remark}
For non-shear steady/traveling flows, stability can only be expected up to translations in the $x_1$-direction. This follows directly from Lemma \ref{lm001}.
Therefore, the stability conclusion in Theorem \ref{thm2} is sharp.
\end{remark}

\begin{remark}
  The smoothness assumption on the perturbed flow can be relaxed. In fact, it suffices to require that the vorticity $\omega$ of the perturbed flow is a continuous admissible map. See Definition \ref{defadmmap} in Section \ref{sec3}.
\end{remark}

Based on Theorem \ref{thm2}, we have the following 
two straightforward corollaries.

 \begin{corollary}
Consider the shear flow with stream function 
\[\bar\psi(x_1,x_2)=a\sin(mx_2+\alpha)+bx_2+c, \quad a,b,c,m,\alpha\in\mathbb R.
\]
  Then the shear flow is stable as in \eqref{stabless1} if $m^2\leq \Lambda_1$.
\end{corollary}

 \begin{corollary}
Consider the steady/traveling flow with stream function 
\[\bar\psi(x_1,x_2)=a 
\cos\left(\frac{2\pi x_1}{L}\right)\sin\left(\frac{\pi x_2}{H}\right)+b
\sin\left(\frac{2\pi x_1}{L}\right)\sin\left(\frac{\pi x_2}{H}\right)+cx_2+d,\quad a,b,c,d \in\mathbb R.\]
If  ${H}/{L}\leq {\sqrt{3}}/{2},$
  then the flow is orbitally stable as in \eqref{stabless2}.
\end{corollary} 
 
 In particular, we have confirmed the existence of orbitally stable \emph{non-shear} flows provided that $H/L\leq \sqrt{3}/2.$

 \subsection{Outline of the proofs}
 To prove Theorems \ref{thm1} and \ref{thm2}, we use a variational approach combined with a compactness argument, which is mainly inspired by \cite{BAR,WZCV,CWZ,Wdisk}. More specifically, we consider the following maximization problem: 
\begin{equation}\label{mmp}
M_{\bar\omega}=\sup_{v\in\mathcal R_{\bar\omega}\cap\mathcal I_{\bar\omega}} E(v),
\end{equation}
where   $\bar\omega$ is as given in Theorem \ref{thm1} or \ref{thm2}, $E$ is defined by \eqref{defoe}, and $\mathcal I_{\bar\omega}$ and $\mathcal R_{\bar\omega}$ are defined by \eqref{defoi} and \eqref{deforc}, respectively. This maximization problem is very natural since the functional $E$ and the constrained set $\mathcal R_{\bar\omega}\cap\mathcal I_{\bar\omega}$  are both invariant under the Euler dynamics. Denote by $\mathcal M_{\bar\omega}$ the set of maximizers of \eqref{mmp}. By employing a compactness argument and utilizing the conservation laws (C1)-(C3) of the Euler equation, we can prove that $\mathcal M_{\bar\omega}$, as well as any of its  \emph{isolated} subsets, is stable under the Euler dynamics. We then study the structure of the set $\mathcal{M}_{\bar\omega}$ and find that: \begin{itemize} 
\item [(1)] $\mathcal{M}_{\bar\omega} = \{\bar\omega\}$ under the conditions of Theorem \ref{thm1}, and
\item [(2)] $\mathcal{O}_{\bar\omega}$ is isolated in $\mathcal{M}_{\bar\omega}$ under the conditions of Theorem \ref{thm2}. 
    \end{itemize}
\noindent  The desired stability then follows immediately. The rigidity results in Theorems \ref{thm1} and \ref{thm2} are derived from the structure of  $\mathcal M_{\bar\omega}$ and the  translational symmetry of the domain.
 
Compared to Arnold's approach, we make better use of the conservation laws of the Euler equation, particularly the isovortical property (C3). The geometric and functional properties of  rearrangement class play an essential role in the analysis.

  \subsection{Organization of this paper}
  
The rest of this paper is organized as follows. In Section \ref{sec2}, we investigate the eigenvalue problem \eqref{lep1}, with a particular focus on the first eigenvalue and first eigenspace. In Section \ref{sec3}, we study the maximization problem \eqref{mmp} and  establish a Burton-type stability criterion for later use. Sections \ref{sec4} and \ref{sec5} are  devoted to the proofs of Theorems \ref{thm1} and \ref{thm2}, respectively.

 \section{Laplace eigenvalue problem}\label{sec2}

In this section, we study the Laplace eigenvalue problem \eqref{lep1}. Define
\begin{equation}\label{defott}
\mathcal Tv:=\mathcal Gv-\mathsf I_{\mathcal Gv}.
\end{equation}
Here and hereafter, we use $\mathsf I_w$ to denote the integral mean of any $w\in L^1(D)$, i.e.,
\[\mathsf I_w=\frac{1}{LH}\int_D wd\mathbf x.\]
Then, for any $\Lambda\in\mathbb R,$ \eqref{lep1} is equivalent to the following operator equation:
\begin{equation}\label{opeq}
v=\Lambda \mathcal Tv,\quad v\in \mathring L^2(D):=\left\{v\in L^2(D)\mid \mathsf I_v=0\right\}.
\end{equation}
In fact,  if $u$ solves \eqref{lep1}, then $v:=-\Delta u$ satisfies \eqref{opeq}. Conversely, if $v$ satisfies \eqref{opeq}, then $u:=\mathcal Tv$ solves \eqref{lep1}.
 By standard elliptic estimates and the Sobolev embedding theorem,
  $\mathcal T$ is a compact operator mapping $\mathring L^2(D)$ into itself.  Moreover, by integration by parts,
  \[\int_Dv_1\mathcal T v_2 d\mathbf x=\int_D(\nabla \mathcal T v_1)\cdot (\nabla \mathcal T v_2) d\mathbf \quad\forall\,v_1,v_2\in\mathring L^2(D),\]
  which implies that $\mathcal T$ is symmetric and  positive definite.  
 Applying the Hilbert-Schmidt theorem, we have the following. 

\begin{lemma}\label{lm021}
\begin{itemize}
  \item [(i)] The eigenvalues of \eqref{lep1} are all real and positive, and can be arranged as  
\[0<\Lambda_1<\Lambda_2<\cdot\cdot\cdot,\quad \lim_{k\to\infty}\Lambda_k=\infty.\]
  \item [(ii)] The eigenspace $\mathbb E_k$ associated with  every $\Lambda_k$ is finite-dimensional.
\item[(iii)]There exists an orthogonal basis $\{w_{k,i}\},$ $k=1,2,\cdot\cdot\cdot,$ $1\leq i\leq d_k$, such that 
    \[\mathbb E_k={\rm span}\left\{w_{k,1},\cdot\cdot\cdot,w_{k,d_k}\right\},\]
    where $d_k$ is the dimension of $\mathbb E_k$.
\end{itemize}
\end{lemma}

Concerning $\Lambda_1$ and $\mathbb E_1,$ we have the following. 
 \begin{lemma}\label{lm022}
  For any $u\in H^1(D)$ such that $\mathsf I_u=0$ and  $u|_{x_2=0}=u|_{x_2=H}={\rm constant},$ it holds that 
 \[\int_D|\nabla u|^2 d\mathbf x\geq \Lambda_1\int_D u^2d\mathbf x,\]
 and the equality is attained if and only if $u\in\mathbb E_1.$ As a consequence, 
   \begin{equation}\label{vtv}
   \int_Dv\mathcal Tv d\mathbf x\geq \Lambda_1\int_D(\mathcal Tv)^2d\mathbf x\quad\forall\,v\in\mathring L^2(D),
   \end{equation}
   and the equality is attained if and only if $v\in\mathbb E_1.$
 \end{lemma}
 \begin{proof}
   The proof is based on the method of eigenfunction expansion, which is almost identical to the proof of Theorem 2(i) in \cite[Section 6.5]{LCE}.  
 \end{proof}

\begin{remark}
We note that \eqref{vtv} implies the following \emph{energy-enstrophy inequality}:
\[
\int_Dv\mathcal Tv d\mathbf x\leq \frac{1}{\Lambda_1}\int_D v^2 d\mathbf x\quad\forall\,v\in\mathring L^2(D),
\] 
with the inequality being an equality if only if  $v\in\mathbb E_1.$ In fact, for any $v\in\mathring L^2(D)$,
\begin{align*}
     \frac{1}{\Lambda_1}\int_D v^2 d\mathbf x  &\geq  2\int_Dv\mathcal Tv d\mathbf x-\Lambda_1\int_D(\mathcal Tv)^2 d\mathbf x \\
     &\geq  2\Lambda_1\int_D(\mathcal Tv)^2 d\mathbf x-\Lambda_1\int_D(\mathcal Tv)^2 d\mathbf x\\
     &=  \Lambda_1\int_D(\mathcal Tv)^2 d\mathbf x ,
   \end{align*}
   where in the first inequality we have used
   \[ \frac{1}{\Lambda_1}a^2+\Lambda_1 b^2\geq 2ab\quad\forall\,a,b\in\mathbb R,\]
    and in the second inequality we have used \eqref{vtv}. 
\end{remark}

\begin{proposition}\label{propevp}
  The set of eigenvalues of \eqref{lep1} is
\[  \left\{\left(\frac{2k\pi}{H}\right)^2\mid k=1,2,3,\cdot\cdot\cdot\right\}\bigcup\left\{ \left(\frac{2n\pi}{L}\right)^2
+\left(\frac{k\pi}{H}\right)^2\mid n,k=1,2,3,\cdot\cdot\cdot\right\}.\]
 The eigenspace $\mathsf E_k$ associated with $\left(\frac{2k\pi}{H}\right)^2$ is given by 
\[ \mathsf E_k:= {\rm span}\left\{\sin \left(\frac{2k\pi x_2}{H}\right),\,\cos\left(\frac{2k\pi x_2}{H}\right)\right\}, \]
and the eigenspace $\mathsf E_{n,k}$ associated with $\left(\frac{2n\pi}{L}\right)^2
+\left(\frac{k\pi}{H}\right)^2$ is given by 
\[\mathsf  E_{n,k}= {\rm span}\left\{\cos\left(\frac{2n\pi x_1}{L}\right)\sin\left(\frac{k\pi x_2}{H}\right),\,
\sin\left(\frac{2n\pi x_1}{L}\right)\sin\left(\frac{k\pi x_2}{H}\right)\right\}.\]
As a corollary, $\Lambda_1$ and $\mathbb E_1$ are determined by $H/L$ as given in \eqref{hl1}-\eqref{hl3}.
\end{proposition}

\begin{proof}
It is clear that every function  in  $\{\mathsf E_k\}_{k=1}^\infty \cup\{\mathsf E_{n,k}\}_{n,k=1}^\infty$ is an eigenfunction of \eqref{lep1}. In view of Lemma \ref{lm021}(iii), it suffices to show that the orthogonal set   $\{\mathsf E_k\}_{k=1}^\infty \cup\{\mathsf E_{n,k}\}_{n,k=1}^\infty$ is complete in $\mathring L^2(D)$. In other words, we need to show that if $v\in\mathring L^2(D)$ satisfies
 \begin{equation}\label{9101}
 \int_Dv(x_1,x_2)\sin \left(\frac{2k\pi x_2}{H}\right) d\mathbf x=0, \quad k=1,2,\cdot\cdot\cdot,
 \end{equation}
  \begin{equation}\label{91011}
  \int_Dv(x_1,x_2)\cos\left(\frac{2k\pi x_2}{H}\right)d\mathbf x=0,\quad k=1,2,\cdot\cdot\cdot,
 \end{equation}
 \begin{equation}\label{9102}
 \int_Dv(x_1,x_2)\sin\left(\frac{2n\pi x_1}{L}\right)\sin\left(\frac{k\pi x_2}{H}\right)d\mathbf x=0,\quad n, k=1,2,\cdot\cdot\cdot,
 \end{equation}
  \begin{equation}\label{91022}
   \int_Dv(x_1,x_2)\cos\left(\frac{2n\pi x_1}{L}\right)\sin\left(\frac{k\pi x_2}{H}\right)d\mathbf x=0,\quad n, k=1,2,\cdot\cdot\cdot, 
 \end{equation}
then $v(\mathbf x)=0$ a.e.  $\mathbf x\in D$.  
We prove this  via the following three steps:
 \begin{itemize}
   \item [(1)] Denote 
   \[\varphi(x_2)=\int_0^Lv(x_1,x_2)dx_1.\]
  By \eqref{9101}, \eqref{91011} and the fact that $\mathsf I_v=0,$  we have that 
 \[\int_0^H\varphi(x_2)\sin \left(\frac{2k\pi x_2}{H}\right)dx_2=0, \quad k=1,2,\cdot\cdot\cdot,\]
   \[\int_0^H\varphi(x_2)\cos \left(\frac{2k\pi x_2}{H}\right)dx_2=0,\quad k=0,1,2,\cdot\cdot\cdot. \]
On the other hand, it is well known that $\left\{\sin \left(\frac{2k\pi x_2}{H}\right)\right\}_{k=1}^\infty\cup \left\{\cos\left(\frac{2k\pi x_2}{H}\right)\right\}_{k=0}^\infty$ is complete in $L^2(0,H)$. Hence $\varphi(x_2)= 0$   a.e.  $x_2\in(0,H)$, or equivalently, 
        \begin{equation}\label{9103}
       \int_0^Lv(x_1,x_2)dx_1= 0\quad{\rm a.e.}\,\,x_2\in(0,H).
        \end{equation}
   \item  [(2)] Denote 
   \[\xi_k(x_1)= \int_0^H v(x_1,x_2) \sin\left(\frac{k\pi x_2}{H}\right)dx_2,\quad k=1,2,\cdot\cdot\cdot.\]
   By \eqref{9102}, \eqref{91022} and \eqref{9103}, we have that 
 \[\int_0^L\xi_k(x_1)\sin \left(\frac{2n\pi x_1}{L}\right)dx_1=0, \quad n=1,2,\cdot\cdot\cdot,\]
   \[\int_0^L\xi_k(x_1)\cos \left(\frac{2n\pi x_1}{L}\right)dx_1=0,\quad n=0,1,2,\cdot\cdot\cdot. \]
In combination with the fact that $\left\{\sin \left(\frac{2n\pi x_1}{L}\right)\right\}_{n=1}^\infty\cup \left\{\cos\left(\frac{2n\pi x_1}{L}\right)\right\}_{n=0}^\infty$ is complete in $L^2(0,L)$,
we deduce that $\xi_k(x_1)=0$  a.e. $x_1\in(0,L).$ So we have proved that for any positive integer $k$,
  \begin{equation}\label{9104}
  \int_0^H v(x_1,x_2) \sin\left(\frac{k\pi x_2}{H}\right)dx_2=0 
  \end{equation}
  a.e. in $(0,L)$.
     \item[(3)]Fix $x_1$ such that \eqref{9104} holds. Extend $v$ as an odd function to $(-H,H)$. Then 
        \begin{equation}\label{9105} 
         \int_{-H}^H v(x_1,x_2) \cos\left(\frac{k\pi x_2}{H}\right)dx_2=0,\quad k=0,1,2,\cdot\cdot\cdot 
         \end{equation}
  by odd symmetry.  Moreover, in view of \eqref{9104}, it holds that
   \begin{equation}\label{91077}
  \int_{-H}^H v(x_1,x_2) \sin\left(\frac{k\pi x_2}{H}\right)dx_2=0 ,\quad k=1,2,\cdot\cdot\cdot 
  \end{equation}
    a.e. in $(0,L)$. 
   Since $\left\{\sin \left(\frac{k\pi x_2}{H}\right)\right\}_{k=1}^\infty\cup \left\{\cos\left(\frac{k\pi x_2}{H}\right)\right\}_{k=0}^\infty$ is complete in $L^2(-H,H)$, we get from \eqref{9105} and \eqref{91077} that 
  \[v(x_1,x_2)=0\quad{\rm a.e. }\,x_2\in(-H,H).\]
  To summarize, we have proved that for a.e. $x_1\in(0,L)$, it holds that $v(x_1,x_2)=0$ a.e. $x_2\in(-H,H)$. Hence  $v(\mathbf x)=0$ a.e.  $\mathbf x\in D$.
        
 \end{itemize}

\end{proof}

\section{Rearrangements, convexity, and stability criterion}\label{sec3}

Throughout this section, let $1<p<\infty$ and $w\in L^p(D)$ be fixed. 
For simplicity, denote $\mathcal R:=\mathcal R_w$ and  
$\mathcal I:=\mathcal I_w$, where $\mathcal R_w$ is the rearrangement class of $w$ defined by \eqref{deforc}, and $\mathcal I_w$ is defined by \eqref{defoi}.

\subsection{Some preliminaries}

We present several lemmas concerning the geometric and functional properties of the functional $E$, defined by  \eqref{defoe}, and the rearrangement class $\mathcal R$ for later use. 
\begin{lemma}\label{lem3001}
 
\begin{itemize}
  \item [(i)] $E$ is well-defined and weakly sequentially continuous in $L^p(D)$.
  \item [(ii)] $E$ is strictly convex, i.e., for any $v_1,v_2\in L^p(D)$ and $\theta\in(0,1)$, it holds that
      \[E(\theta v_1+(1-\theta))\leq \theta E(v_1)+(1-\theta)E(v_2),\]
      and the equality is attained if and only $v_1=v_2$.
\end{itemize}
 \end{lemma}
 
 \begin{proof}
 Item (i) follows from standard elliptic estimates and the Sobolev embedding theorem. Item (ii) follows from the fact  that $\mathcal G$ is  symmetric, i.e.,
 \begin{equation}\label{ii318}
 \int_Dv_1\mathcal Gv_2 d\mathbf x=\int_Dv_2\mathcal Gv_1 d\mathbf x,\quad \forall\,v_1,v_2\in L^p(D),
 \end{equation}
 and positive-definite, i.e., 
  \begin{equation}\label{ii319}
  \int_Dv \mathcal Gv  d\mathbf x\geq 0\quad\forall\,v\in L^p(D),
  \end{equation}
  with the inequality being an equality if and only if $v=0.$ Note that \eqref{ii318} and \eqref{ii319} can be easily verified by integration by parts.
  \end{proof}

For a nonempty convex set $C$ of a real vector space $X$, a point $x\in C$ is called an \emph{extreme point} of $C$ if there do \emph{not} exist  $y,z\in C$ and $\theta\in(0,1)$ such that   $y\neq z$ and  $x=\theta y+(1-\theta)z.$

\begin{lemma}\label{lem312}
Denote by $\bar{\mathcal R}$ the \emph{weak} closure of $\mathcal R$ in $L^p(D)$.  Then
\begin{itemize}
\item[(i)] $\bar{\mathcal R}$ is convex and weakly compact in $L^p(D)$,  and the set of extremum points of $\bar{\mathcal R}$ is exactly $\mathcal R$; 
  \item [(ii)]  $\bar{\mathcal R}\cap\mathcal I$ is convex and weakly compact in $L^p(D)$,   and the set of extremum points of $\bar{\mathcal R}\cap\mathcal I$ is exactly $\mathcal R\cap\mathcal I$. 
\end{itemize}
 
\end{lemma}
\begin{proof}
 Item (i) follows from  \cite[Theorem 6]{BMA} or in \cite[Lemma 2.2]{BHP}. Item (ii) is a special case of  \cite[Lemma 4.4]{BMcL} or \cite[Lemma 5]{BRy}.
\end{proof}

\begin{lemma}\label{lem313}  
Fix $w'\in L^p(D)$. Denote $\mathcal R':=\mathcal R_{w'}$. 
  Then for any $v\in\mathcal R$, there exists $v'\in \mathcal R'$ such that
\begin{equation}\label{kd01}
\|v-v'\|_{L^p(D)}=\min_{v_1\in\mathcal R, v_2\in\mathcal R'}\|v_1-v_2\|_{L^p(D)}.
\end{equation}
\end{lemma}
\begin{proof}
See Lemma 2.3 in \cite{Bata}.
\end{proof}

\begin{lemma}\label{lem314}  
Let $q=p/(p-1)$ be the H\"older conjugate of $p$.  Fix $\xi\in L^q(D)$. 
Suppose that $\tilde w\in\mathcal R$ satisfies 
\[\tilde w(\mathbf x)=\phi(\xi(\mathbf x))\,\,\,{\rm a.e. }\,\,\mathbf x\in D\]
for some increasing function $\phi:\mathbb R\to\mathbb R\cup\{\pm\infty\}$. Then $\tilde  w$ is the unique maximizer of   $<\xi,\cdot>$ relative $\bar{\mathcal R}.$ Here
\[<\xi, v>=\int_D\xi vd\mathbf x\quad\forall\,v\in L^p(D).\]
\end{lemma}
\begin{proof}
See Theorem 3 in \cite{BMA} or Lemma 2.4 in \cite{BMcL}.
\end{proof}

\subsection{A variational problem involving rearrangements}
 
  Consider the following maximization problem:
\begin{equation}\label{mmp0}
M=\sup_{v\in\mathcal R\cap\mathcal I}E(v).
\end{equation}
Denote by $\mathcal M$ the set of maximizers of \eqref{mmp0}, i.e.,
\[\mathcal M:=\left\{v\in \mathcal R\cap\mathcal I\mid E(v)=M\right\}.\]

To analyze \eqref{mmp0}, we need the following lemma, which is proved based on convexity.
\begin{lemma}\label{btrbar}
The maximum of $E$ relative to $\bar{\mathcal R}\cap\mathcal I$ is attained, and every maximizer must belong to ${\mathcal R}\cap\mathcal I$. In particular, it holds that
\[\sup_{v\in {\mathcal R}\cap\mathcal I}E(v)=\sup_{v\in\bar{\mathcal R}\cap\mathcal I}E(v).\]
\end{lemma}
\begin{proof}
By Lemma \ref{lem3001}(i) and Lemma \ref{lem312}(ii), there exists  $\tilde v\in\bar{\mathcal R}\cap\mathcal I$ such that
\[E(\tilde v)=\sup_{v\in\bar{\mathcal R}\cap\mathcal I}E(v).\]
Our task is to show that $\tilde v\in\bar{\mathcal R}\cap\mathcal I.$ By Lemma \ref{lem312}(ii), it suffices to show that $\tilde v$ is an extreme point of $\bar{\mathcal R}\cap\mathcal I$. To this end, suppose that there exist  $v_1,v_2\in\mathcal R\cap\mathcal I$ and $\theta\in(0,1)$ such that
$\tilde v =\theta v_1+(1-\theta)v_2.$ By the choice of $\tilde v,$ it holds that
\begin{equation}\label{3191}
E(v_1)\leq E(\tilde v),\quad E(v_2)\leq E(\tilde v).
\end{equation}
On the other hand, by Lemma \ref{lem3001}(ii), we have that
\begin{equation}\label{3192}
E(\tilde v)\leq \theta E(v_1)+(1-\theta)E(v_2),
\end{equation}
\begin{equation}\label{3193}
E(\tilde v)= \theta E(v_1)+(1-\theta)E(v_2)\quad\Longleftrightarrow\quad v_1=v_2.
\end{equation}
Combining \eqref{3191}-\eqref{3193}, we deduce that $v_1=v_2$, which means that $\tilde v$ is indeed an extreme point of $\bar{\mathcal R}\cap\mathcal I$. This completes the proof.
\end{proof}

\begin{proposition}\label{propcomt1}
Suppose that $\{v_n\}\subset\mathcal R$ satisfies
\[\lim_{n\to\infty} E (v_n)=M,\quad \lim_{n\to\infty} I(v_n)=I(w).\]
Then, up to a subsequence,  $\{v_{n}\}$  converges to some $\tilde v\in\mathcal M$ strongly in $L^p(D)$ as $n\to\infty.$ In particular, the set $\mathcal M$ is a nonempty and compact subset of $L^p(D)$.
\end{proposition}

\begin{proof} 
  Since $\{v_n\}$ is obviously bounded in $L^p(D)$, we can assume that, up to a subsequence,  $\{v_n\}$  converges \emph{weakly} to some $\tilde v\in\bar{\mathcal R}$  in $L^p(D)$ as $n\to\infty$. It is clear that $I(\tilde v)=I(w)$, hence  $\tilde v\in\bar{\mathcal R}\cap\mathcal I.$
   Moreover, since $E$ is weakly sequentially continuous (cf. Lemma \ref{lem3001}(i)), we have that
  \begin{equation}\label{em01}
  E(\tilde v)=\lim_{n\to\infty} E (v_n)=M.
  \end{equation}
 Applying  Lemma \ref{btrbar}, we deduce that $\tilde v\in\mathcal R\cap\mathcal I,$ and thus $\tilde v\in\mathcal M.$

 To prove strong convergence in $L^p(D)$, observe that for any positive integer $n$,
 \[\|v_{n}\|_{L^p(D)}=\|w\|_{L^p(D)}=\|\tilde v\|_{L^p(D)}.\]
 The second equality holds since we have established that $\tilde v\in\mathcal R.$ Strong convergence then follows from uniform convexity (cf. Proposition 3.32 in \cite[\S 3.7]{Bre}).
  
\end{proof}

\subsection{A Burton-type stability criterion}

To begin with, we introduce the concept of admissible maps in order to illustrate the requirements for the perturbed solutions.

\begin{definition}\label{defadmmap}
A map $\zeta:\mathbb R\to L^p(D)$ is called an \emph{admissible map} if for all $t\in\mathbb R,$ 
\[E(\zeta(t))=E(\zeta(0)),\quad I(\zeta(t))=I(\zeta(0)),\quad \zeta(t)\in\mathcal R_{\zeta(0)}.\]
If, in addition,  $\zeta$ is continuous, we call it  a  \emph{continuous admissible map}.
\end{definition}

By the conservation laws (C1)-(C3),  for any sufficiently smooth Euler flow with vorticity $\omega(t,\mathbf x)$, $\zeta(t):=\omega(t,\cdot)$ is a continuous admissible map.

 The purpose of this subsection is to prove the following Burton-type stability criterion, which is inspired by Burton's work \cite{BAR} (see also \cite{WMA, WZCV, Wdisk}).
\begin{theorem}\label{bsc}
Let $\mathcal M$ be as given in the previous subsection. Then:
 \begin{itemize}
\item[(i)] For any $\varepsilon>0$, there exists some $\delta>0,$ such that for any admissible map $\zeta(t)$, 
  \[\min_{v\in\mathcal M}\|\zeta(0)-v\|_{L^p(D)}<\delta\quad\Longrightarrow\quad \min_{v\in\mathcal M}\|\zeta(t)-v\|_{L^p(D)}<\varepsilon\quad\forall\,t\in\mathbb R.\]
  \item[(ii)]  Let $\mathcal S$ be a nonempty proper subset of $\mathcal M.$  Suppose that  $\mathcal S$ is isolated in $\mathcal M,$ i.e., 
     \[\min_{v_1\in\mathcal S,\,v_2\in\mathcal M\setminus\mathcal S}\|v_1-v_2\|_{L^p(D)} >0.
       \]
 Then, for any $\varepsilon>0$, there exists some $\delta>0$ such that for any continuous admissible map $\zeta(t)$, 
  \[\min_{v\in\mathcal S}\|\zeta(0)-v\|_{L^p(D)}<\delta\quad\Longrightarrow\quad \min_{v\in\mathcal S}\|\zeta(t)-v\|_{L^p(D)}<\varepsilon\quad\forall\,t\in\mathbb R.\]
  \end{itemize}
\end{theorem}

\begin{proof} 
  Since  (ii) can be deduced from (i) via a standard continuity argument, we only give the proof of (i).

 It suffices to show that for any sequence of admissible maps $\{\zeta_n(t)\}$ and any sequence of times $\{t_n\}\subset\mathbb R$,  if
\begin{equation}\label{tt09}
\lim_{n\to\infty}\|\zeta_n(0)-\hat v\|_{L^p(D)}=0
\end{equation}
for some $\hat v\in\mathcal M$,
then, up to a subsequence, there exists some $\tilde v\in\mathcal M$ such that
\begin{equation}\label{tt010}
\lim_{n\to\infty}\|\zeta_n(t_n)-\tilde v\|_{L^p(D)}= 0.
\end{equation}
By \eqref{tt09}, we have that $\lim_{n\to\infty}E(\zeta_n(0))=E(\hat v)=M$. Using the fact that $E$ is conserved for admissible maps, we have that
\begin{equation}\label{gx21}
\lim_{n\to\infty}E(\zeta_n(t_n))=M.
\end{equation}

Next we construct a sequence of ``followers" $\{\eta_n\}\subset\mathcal R$ related to the sequence $\{\zeta_n(t_n)\}$. To be specific, for each positive integer $n$, we can take some  $\eta_n\in\mathcal R$ such that
\[\|\eta_n-\zeta_n(t_n)\|_{L^p(D)}=\min\left\{\|v_1-v_2\|_{L^p(D)}\mid v_1\in\mathcal R,\,v_2\in\mathcal R_{\zeta_n(t_n)}\right\}.\]
Note that the existence of such $\eta_n$ is ensured by Lemma \ref{lem313}. 
 In particular, $\eta_n$ satisfies
\begin{equation}\label{ty12}
 \|\eta_n-\zeta_n(t_n)\|_{L^p(D)}\leq \|\hat v-\zeta_n(0)\|_{L^p(D)},
\end{equation}
since $\hat v\in\mathcal R$ and $\zeta_n(0)\in\mathcal R_{\zeta_n(t_n)}$.
Combining \eqref{tt09} and \eqref{ty12}, we get
\begin{equation}\label{tt121}
\lim_{n\to\infty}\|\eta_n-\zeta_n(t_n)\|_{L^p(D)} =0,
\end{equation}
and thus
\begin{equation}\label{tt08}
\lim_{n\to\infty} E(\eta_n)= M
\end{equation}
by \eqref{gx21}. Besides, it is clear that 
\begin{equation}\label{tt0y8}
\lim_{n\to\infty} I(\eta_n)= \lim_{n\to\infty} I(\zeta_n(t_n))=\lim_{n\to\infty} I(\zeta_n(0))=I(\hat v)=I(w).
\end{equation}
To summarize, we have constructed  a sequence  $\{\eta_n\}\subset\mathcal R$ such that \eqref{tt08} and \eqref{tt0y8} hold.
By  Proposition \ref{propcomt1},  $\eta_n$ converges, up to a subsequence, to some $\tilde v\in\mathcal M$ strongly in $L^p(D)$  as $n\to\infty$, which together with \eqref{tt121} yields \eqref{tt010}.

\end{proof}

\subsection{Application: stability of shear flows with monotone vorticity}

In the following, we present an application of Theorem \ref{bsc}, although it is not directly related to the main results of this paper.

\begin{theorem}
Suppose that $\bar\omega\in L^p(D)$ satisfies
\[\bar\omega=\phi(x_2)\quad\mbox{a.e.\,\,\,in}\,\,\,D \]
for some increasing or decreasing function $\phi:\mathbb R\to\mathbb R\cup\{\pm\infty\}$. The $\bar\omega$ is stable under the Euler dynamics in the following sense: for any $\varepsilon>0,$ there exists some $\delta>0,$ such that for any sufficiently smooth Euler flow with vorticity $\omega(t,\mathbf x)$,  it holds that
    \[
   \|\omega(0,\cdot)-\bar\omega\|_{L^p(D)} <\delta \quad\Longrightarrow\quad \|\omega(t,\cdot)- \bar\omega  \|_{L^p(D)}<\varepsilon\quad \forall\,t\in\mathbb R.
\]
\end{theorem}
\begin{proof}
Without loss of generality, we assume that $\phi$ is increasing; otherwise, we can consider $-\bar{\omega}$.
Applying Lemma \ref{lem314} (choosing $\xi=x_2$ and $\tilde w=\bar\omega$ therein), we see that $\bar\omega$ is the unique maximizer of $I$ relative to $\mathcal R_{\bar\omega}$, which implies that
\[\mathcal R_{\bar\omega}\cap \mathcal I_{\bar\omega}=\{\bar\omega\}.\]
In particular, $\bar\omega$ is the unique maximizer of $E$ relative to $\mathcal R_{\bar\omega}\cap\mathcal I_{\bar\omega}$. The desired stability then follows from Theorem \ref{bsc}.
\end{proof}

\section{Proof of Theorem \ref{thm1}}\label{sec4}

We begin with some preliminary discussions before presenting the proof.
Redefine the values of $g$ outside the interval $\left[\min_{\bar D}g'(\bar\psi-\lambda x_2),\max_{\bar D}g'(\bar\psi-\lambda x_2)\right]$ such that  
 \begin{equation}\label{aspg1}
  g' \equiv c_0\quad\forall\,s\in \left(-\infty,\min_{\bar D}g'(\bar\psi-\lambda x_2)-1\right]\cup \left[\max_{\bar D}g'(\bar\psi-\lambda x_2)+1,+\infty\right) 
 \end{equation}
for some positive number $c_0,$ and
 \begin{equation}\label{aspg2}
 0\leq g'(s)\leq \Lambda_1-\epsilon\quad\forall\,s\in\mathbb R
 \end{equation}
for some sufficiently small positive number $\epsilon$. This can be achieved by repeating the proof of  \cite[Lemma 2.4]{WTAMS}. By \eqref{aspg1},
\begin{equation}\label{texp}
G(s+\tau)\leq G(s)+ g(s)\tau+\frac{1}{2}(\Lambda_1-\epsilon)\tau^2\quad\forall\,
s,\tau\in\mathbb R.
\end{equation}
Let $G$ be an antiderivative of $g$ (for example, one can take $G(s)=\int_0^sg(\tau) d\tau$). Denote by $\hat G$ the Legendre transform  of $G$, i.e.,
\[\hat G(s):=\sup_{\tau\in\mathbb R}\left(s\tau-G(\tau)\right).\]
The condition \eqref{aspg1} ensures that   $\hat G$
 is a real-valued, locally Lipschitz function (cf. \cite[Lemma 2.3]{WTAMS}). It is clear according to the definition of Legendre transform that
\begin{equation}\label{hatg}
\hat G(s)+G(\tau)\geq s\tau\quad\forall\,s,\tau\in\mathbb R,
\end{equation}
and the inequality is an equality if and only if $s=g(\tau)$. Note that the introduction of the Legendre transform here is mainly inspired by Wolansky and Ghil's papers \cite{WG1,WG2}.

Now we turn to the proof. Recall that $\mathcal M_{\bar\omega}$ is the set of maximizers of \eqref{mmp}.

 \begin{proposition}\label{keyprop1}
Suppose that $\bar\omega$ satisfies the conditions of Theorem \ref{thm1}. Then  
\begin{equation}\label{mstr}
\mathcal M_{\bar\omega}=\{\bar\omega\}.
\end{equation}
\end{proposition}
\begin{proof}

Fix  $\varrho\neq 0$ such that $\bar\omega+\varrho\in\mathcal R_{\bar\omega}\cap\mathcal I_{\bar\omega}$. Our aim is to show that $E(\bar\omega)>E(\bar\omega+\varrho).$
By a straightforward computation,
\begin{equation}\label{emsm}
E (\bar\omega)-E (\bar\omega+\varrho)
=-\int_D\varrho\mathcal G\bar\omega d\mathbf x-\frac{1}{2}\int_D\varrho\mathcal G\varrho d\mathbf x.
\end{equation}
 To proceed, we observe that $\varrho$ satisfies the following  constraints:
\begin{equation}\label{rhoc}
  \int_D\varrho d\mathbf x=0,\quad I(\varrho)=0,\quad \int_D\hat G(\bar\omega+\varrho) d\mathbf x= \int_D\hat G(\bar\omega) d\mathbf x.
\end{equation}
Based on \eqref{rhoc}, we compute as follows (recall that $\mathcal T$ is defined by \eqref{defott}):
\begin{equation}\label{casicom}
\begin{split}
0=&\int_D\hat G(\bar\omega+\varrho) d\mathbf x-\int_D\hat G(\bar\omega) d\mathbf x\\
\geq &\int_D(\bar\omega+\varrho)(\bar\psi-\lambda x_2+\mathcal T\varrho) d\mathbf x-\int_DG(\bar\psi-\lambda x_2+\mathcal T\varrho) d\mathbf x-\int_D\bar\omega(\bar\psi-\lambda x_2)d\mathbf x\\
&+\int_DG(\bar\psi-\lambda x_2)d\mathbf x\\
=&\int_D\bar\omega\mathcal T\varrho+\varrho\bar\psi+\varrho\mathcal T\varrho d\mathbf x-\int_DG(\bar\psi-\lambda x_2+\mathcal T\varrho)-G(\bar\psi-\lambda x_2) d\mathbf x\\
\geq& \int_D\bar\omega\mathcal T\varrho+\varrho\bar\psi +\varrho\mathcal T\varrho  d\mathbf x-\int_Dg(\bar\psi-\lambda x_2)\mathcal T\varrho +\frac{1}{2}(\Lambda_1-\epsilon) (\mathcal T\varrho)^2 d\mathbf x\\
=& \int_D\varrho\bar\psi +\varrho\mathcal T\varrho  d\mathbf x-\frac{1}{2}\int_D (\Lambda_1-\epsilon) (\mathcal T\varrho)^2 d\mathbf x\\
=& \int_D\varrho\mathcal G\bar\omega  +\varrho\mathcal T\varrho  d\mathbf x-\frac{1}{2}\int_D (\Lambda_1-\epsilon) (\mathcal T\varrho)^2 d\mathbf x,
\end{split}
\end{equation}
where we have used \eqref{hatg} in the first inequality and \eqref{texp} in the second inequality.
 So we have proved that 
 \begin{equation}\label{keyineq}
 -\int_D\varrho \mathcal G\bar\omega d\mathbf x -\frac{1}{2}\int_D\varrho\mathcal T\varrho d\mathbf x \geq \frac{1}{2}\int_D\varrho\mathcal T\varrho d\mathbf x-\frac{1}{2}\int_D (\Lambda_1-\epsilon) (\mathcal T\varrho)^2 d\mathbf x.
 \end{equation}
 Inserting \eqref{keyineq} into \eqref{emsm} and applying Lemma \ref{lm022}, we get  
\[
E (\bar\omega)-E (\bar\omega+\varrho)\geq \frac{1}{2}\int_D\varrho\mathcal T\varrho-(\Lambda_1-\epsilon) (\mathcal T\varrho)^2d\mathbf x  \geq \frac{1}{2}\epsilon\int_D  (\mathcal T\varrho)^2d\mathbf x>0.\]
Hence the proof is complete.
 \end{proof}

\begin{proof}[Proof of Theorem \ref{thm1}]
The stability assertion (ii) follows directly as a corollary of Theorem \ref{bsc} and Proposition \ref{keyprop1}. It remains to prove the rigidity assertion (i). Observe that $E$, $\mathcal R_{\bar\omega}$ and $\mathcal I_{\bar\omega}$ are all invariant under translations along the $x_1$-direction, i.e., for any $\alpha\in\mathbb R,$
  \[E(v)=E(v(\cdot+\alpha\mathbf e_1)),\quad
   v\in\mathcal R_{\bar\omega}\,\,\Longleftrightarrow\,\, v(\cdot+\alpha\mathbf e_1)\in \mathcal R_{\bar\omega},\quad v\in\mathcal I_{\bar\omega}\,\,\Longleftrightarrow\,\,  v(\cdot+\alpha\mathbf e_1)\in \mathcal I_{\bar\omega}.\]
 Hence $\mathcal M_{\bar\omega}$ is also invariant under translations along the $x_1$-direction, i.e.,
  \begin{equation}\label{minva}
  v\in\mathcal M_{\bar\omega}\quad\Longrightarrow\quad v(\cdot+\alpha\mathbf e_1)\in \mathcal M_{\bar\omega}\quad\forall\,\alpha\in\mathbb R.
  \end{equation}
  The desired result then follows immediately from Proposition \ref{keyprop1}.
\end{proof}

\section{Proof of Theorem \ref{thm2}} \label{sec5}

\begin{proposition}\label{keyprop2}
Suppose that $\bar\omega$ satisfies the conditions of Theorem \ref{thm2}. Then  
\[\mathcal M_{\bar\omega}=\left\{ \bar\omega+\varrho\mid  \varrho+\bar\omega\in  \mathcal R_{\bar\omega},\, \varrho\in\mathbb E_1,\, I(\varrho)=0,\,\mathsf I_{\bar\omega}\mathsf I_{\mathcal G\varrho}=0\right\}.\]
In particular, it holds that 
 \begin{equation}\label{inpr}
 \mathcal M_{\bar\omega}\subset(\bar\omega+\mathbb E_1)\cap\mathcal R_{\bar\omega}.
 \end{equation}
 \end{proposition}
\begin{proof}
Fix $\varrho$ such that $\bar\omega+\varrho\in\mathcal R_{\bar\omega}\cap \mathcal I_{\bar\omega}$. As in Section \ref{sec4}, we can assume that $g$ satisfies \eqref{aspg1} and \eqref{aspg2} with $\epsilon=0$.
 Repeating the arguments \eqref{emsm}-\eqref{keyineq}, we can prove that
\begin{equation}\label{egeq2}
E (\bar\omega)-E (\bar\omega+\varrho)\geq \frac{1}{2}\int_D\varrho\mathcal T\varrho-\Lambda_1   (\mathcal T\varrho)^2d\mathbf x\geq 0.
\end{equation}
Hence $\bar\omega\in\mathcal M_{\bar\omega}$. Moreover, by Lemma \ref{lm022}, if $E (\bar\omega)=E (\bar\omega+\varrho)$, then necessarily $\varrho\in\mathbb E_1.$

To complete the proof, it suffices to show that for any $\varrho$ satisfying $\bar\omega+\varrho\in  \mathcal R_{\bar\omega}\cap\mathcal I_{\bar\omega}$,  
  \[
  E (\bar\omega)=E (\bar\omega+\varrho)\quad\Longleftrightarrow\quad
 \varrho\in\mathbb E_1\,\, {\rm and }\,\,\mathsf I_{\bar\omega}\mathsf I_{\mathcal G\varrho}=0.\]
First we prove ``$\Longrightarrow$". Suppose that $E(\bar\omega)=E(\bar\omega+\varrho)$. Then
 $\varrho\in\mathbb E_1$, and  
\begin{equation}\label{csble}
\begin{split}
0=&E(\bar\omega)-E(\bar\omega+\varrho)\\
 =&-\int_D\bar\omega \mathcal G\varrho d\mathbf x-\frac{1}{2}\int_D\varrho\mathcal G\varrho d\mathbf x \\
=&-\int_D\bar\omega \mathcal T\varrho d\mathbf x-\frac{1}{2}\int_D\varrho\mathcal T\varrho d\mathbf x -\mathsf I_{\mathcal G\varrho}\int_D\bar\omega   d\mathbf x\\
 =&-\frac{1}{\Lambda_1}\int_D\bar\omega  \varrho d\mathbf x-\frac{1}{2}\int_D\varrho\mathcal T\varrho d\mathbf x -LH \mathsf I_{\bar\omega} \mathsf I_{\mathcal G\varrho} \\
=& \frac{1}{2\Lambda_1}\int_D   \varrho^2 d\mathbf x-\frac{1}{2}\int_D\varrho\mathcal T\varrho d\mathbf x -LH \mathsf I_{\bar\omega} \mathsf I_{\mathcal G\varrho} \\
 =& -LH \mathsf I_{\bar\omega}\mathsf I_{\mathcal G\varrho},
\end{split}
\end{equation}
which yields $\mathsf I_{\bar\omega} \mathsf I_{\mathcal G\varrho}=0$.
Note that we have used the following equality in \eqref{csble}:
\[- \int_D\bar\omega  \varrho d\mathbf x=\frac{1}{2}\int_D   \varrho^2 d\mathbf x,\]
which follows from  $\|\bar\omega+\varrho\|_{L^2(D)}=\|\bar\omega\|_{L^2(D)}.$
Conversely, if $\varrho\in\mathbb E_1$ and $\mathsf I_{\bar\omega}\mathsf I_{\mathcal G\varrho}=0$, then \eqref{csble} still holds, which gives  $E(\bar\omega)=E(\bar\omega+\varrho)$.
 The proof is complete.
\end{proof}

\begin{lemma}\label{lem422}
There exists some $\bar\omega^e\in\mathbb E_1$ and some  $\bar\omega^s$  depending only on $x_2$ such that
 \[\bar\omega=\bar\omega^e+\bar\omega^s.\]
 
 \end{lemma}
 \begin{proof}
 Since $\mathcal M_{\bar\omega}$ is invariant under translations along the $x_1$-direction (cf. \eqref{minva}), it holds that \[\bar\omega(\cdot+\alpha\mathbf e_1)\in\mathcal M_{\bar\omega}\quad\forall\,  \alpha\in\mathbb R.\]
  Taking into account   \eqref{inpr}, we have that 
   \[\bar\omega(\cdot+\alpha\mathbf e_1)-\bar\omega\in\mathbb E_1\quad\forall\,\alpha\in\mathbb R,
    \]
  which implies that $\partial_{x_1}\bar\omega\in\mathbb E_1.$

To proceed, we distinguish three cases:

\noindent {\bf Case 1:} If $H/L>\sqrt{3}/2$, then, by \eqref{hl1},
  \begin{equation}\label{fenj1}
  \bar\omega =x_1\left[a\sin\left(\frac{2\pi x_2}{H}\right)+b\cos\left(\frac{2\pi x_2}{H}\right)\right]+\bar\omega^s
  \end{equation}
  for some $a,b\in\mathbb R$ and some $\bar\omega^s$ that depends only on $x_2$. On the other hand, we know that $\bar\omega(\cdot+\alpha\mathbf e_1)\in\mathcal R_{\bar\omega}$ for any $\alpha\in\mathbb R$,
  which implies that
  \begin{equation}\label{fenj2}
  \sup_{\alpha\in\mathbb R}\|\bar\omega(\cdot+\alpha\mathbf e_1)\|_{L^\infty(D)}<\infty.
  \end{equation}
Combining \eqref{fenj1} and \eqref{fenj2}, we deduce that $a=b=0.$

\noindent {\bf Case 2:} If $H/L<\sqrt{3}/2$, then, by \eqref{hl2},
  \[\bar\omega = a\cos\left(\frac{2\pi x_1}{L}\right)\sin\left(\frac{\pi x_2}{H}\right)+b\sin\left(\frac{2\pi x_1}{L}\right)\sin\left(\frac{\pi x_2}{H}\right) +\bar\omega^s\]
  for some $a,b\in\mathbb R$ and some $\bar\omega^s$  that depends  only on $x_2$.

\noindent {\bf Case 3:} If $H/L=\sqrt{3}/2$, then, by \eqref{hl3},
  \[\bar\omega = x_1\left[a\sin\left(\frac{2\pi x_2}{H}\right)+b\cos\left(\frac{2\pi x_2}{H}\right)\right]+c\cos\left(\frac{2\pi x_1}{L}\right)\sin\left(\frac{\pi x_2}{H}\right)+d\sin\left(\frac{2\pi x_1}{L}\right)\sin\left(\frac{\pi x_2}{H}\right) +\bar\omega^s\]
  for some $a,b,c,d\in\mathbb R$ and some $\bar\omega^s$  that depends  only on $x_2$. Repeating the argument as in {\bf Case 1}, we get $a=b=0.$ 
 
Therefore the desired result holds in all three cases, which completes the proof.
  
    \end{proof}

\begin{lemma}\label{lem423}
  $\mathcal O_{\bar\omega}$ is an isolated subset of $\mathcal M_{\bar\omega}.$
\end{lemma}

\begin{proof}

We only give the proof for the case $H/L=\sqrt{3}/2$. The proofs for the other two cases are similar (and simpler).  Recall that when $H/L=\sqrt{3}/2$,
\[ \mathbb E_1={\rm span}\left\{\sin\left(\frac{2\pi x_2}{H}\right),\,\cos\left(\frac{2\pi x_2}{H}\right),\,\cos\left(\frac{2\pi x_1}{L}\right)\sin\left(\frac{\pi x_2}{H}\right),\,
\sin\left(\frac{2\pi x_1}{L}\right)\sin\left(\frac{\pi x_2}{H}\right)\right\}.\]
Fix $v\in\mathcal M_{\bar\omega}\setminus \mathcal O_{\bar\omega}$. Our aim is to show that \emph{ $\|v-\bar\omega\|_{L^p(D)}$ has a positive lower bound  depending only on $\bar\omega$.} 

By Proposition \ref{keyprop2} and Lemma \ref{lem422}, we can write
\[\bar\omega=\bar\omega^s+s\sin\left(\frac{2\pi x_2}{H} \right)+ a\cos\left(\frac{2\pi x_2}{H} \right)+ b\cos\left(\frac{2\pi x_1}{L}+ \beta\right)\sin\left(\frac{\pi x_2}{H}\right),\]
\[v=\bar\omega^s+s'\sin\left(\frac{2\pi x_2}{H} \right)+ a'\cos\left(\frac{2\pi x_2}{H} \right)+ b'\cos\left(\frac{2\pi x_1}{L}+ \beta'\right)\sin\left(\frac{\pi x_2}{H}\right),\]
where $s,a,b,\beta,s',a',b',\beta'\in\mathbb R$. Without loss of generality, we can assume that $b,b'\geq 0.$
Since   $I(\bar\omega)=I(v)$ and  
\[I\left(\sin\left(\frac{2\pi x_2}{H} \right)\right)\neq 0,\quad I\left(\cos\left(\frac{2\pi x_2}{H} \right)\right)= I\left(\cos\left(\frac{2\pi x_1}{L}\right)\sin\left(\frac{\pi x_2}{H}\right) \right)= 0,\]
 we get
 \begin{equation}\label{ssp}
 s=s'.
 \end{equation}  
 Using $v\in\mathcal R_{\bar\omega}$, we have that
 \begin{equation}\label{intmm0}
 \int_D\bar\omega^md\mathbf x=\int_Dv^md\mathbf x
 \end{equation}
 for any positive integer $m$.   Denote
 \[\tilde\omega^s:=\bar\omega^s+s\sin\left(\frac{2\pi x_2}{H} \right).\]
 Then \eqref{intmm0} becomes
  \begin{equation}\label{intmm}
  \begin{split}
 &\int_D\left[\tilde\omega^s+ a\cos\left(\frac{2\pi x_2}{H} \right)+ b\cos\left(\frac{2\pi x_1}{L}\right)\sin\left(\frac{\pi x_2}{H}\right)\right]^md\mathbf x\\
=&\int_D\left[\tilde\omega^s+ a'\cos\left(\frac{2\pi x_2}{H} \right)+ b'\cos\left(\frac{2\pi x_1}{L}\right)\sin\left(\frac{\pi x_2}{H}\right)\right]^md\mathbf x.
 \end{split}
 \end{equation}
Taking $m=2$ in \eqref{intmm}, we get 
   \begin{equation}\label{req1}
   \begin{split}
   &2a\int_D\tilde\omega^s\cos \left(\frac{2\pi x_2}{H}\right)d\mathbf x+ a^2\int_D\cos^2\left(\frac{2\pi x_2}{H}\right)d\mathbf x+b^2\int_D \cos^2\left(\frac{2\pi x_1}{L} \right)\sin^2\left(\frac{\pi x_2}{H}\right)d\mathbf x\\
   =&2a'\int_D\tilde\omega^s\cos \left(\frac{2\pi x_2}{H}\right)d\mathbf x+ a'^2\int_D\cos^2\left(\frac{2\pi x_2}{H}\right)d\mathbf x+b'^2\int_D \cos^2\left(\frac{2\pi x_1}{L} \right)\sin^2\left(\frac{\pi x_2}{H}\right)d\mathbf x,
   \end{split}
   \end{equation}
 which can be simplified as
\[ 2 a^2 +  b^2+pa =2 a'^2 +  b'^2+pa',\]
where
\[ p:=\frac{8}{LH} \int_D\tilde\omega^s\cos \left(\frac{2\pi x_2}{H}\right)d\mathbf x.\]
  Similarly, taking $m=3$ in \eqref{intmm} gives 
  \begin{equation}\label{req2}
  \begin{split}
   &a^2\int_D\tilde\omega^s\cos^2 \left(\frac{2\pi x_2}{H}\right)d\mathbf x+b^2\int_D\tilde\omega^s\cos^2 \left(\frac{2\pi x_1}{L}\right)\sin^2\left(\frac{\pi x_2}{H}\right)d\mathbf x+a\int_D(\tilde\omega^s)^2\cos  \left(\frac{2\pi x_2}{H}\right)d\mathbf x\\
  &+  ab^2\int_D\cos  \left(\frac{2\pi x_2}{H}\right)\cos^2 \left(\frac{2\pi x_1}{L}\right)\sin^2\left(\frac{\pi x_2}{H}\right)d\mathbf x\\
   =&a'^2\int_D\tilde\omega^s\cos^2 \left(\frac{2\pi x_2}{H}\right)d\mathbf x+b'^2\int_D\tilde\omega^s\cos^2 \left(\frac{2\pi x_1}{L}\right)\sin^2\left(\frac{\pi x_2}{H}\right)d\mathbf x+a'\int_D(\tilde\omega^s)^2\cos  \left(\frac{2\pi x_2}{H}\right)d\mathbf x\\
  &+  a'b'^2\int_D\cos  \left(\frac{2\pi x_2}{H}\right)\cos^2 \left(\frac{2\pi x_1}{L}\right)\sin^2\left(\frac{\pi x_2}{H}\right)d\mathbf x,
  \end{split}
  \end{equation} 
 which can be simplified as
  \[-ab^2+qa^2+rb^2+\mu a=-a'b'^2+qa'^2+rb'^2+\mu a',\]
  where
  \[q:= \frac{8}{LH}  \int_D\tilde\omega^s\cos^2 \left(\frac{2\pi x_2}{H}\right)d\mathbf x,\]
   \[r:= \frac{8}{LH} \int_D\tilde\omega^s\cos^2 \left(\frac{2\pi x_1}{L}\right)\sin^2\left(\frac{\pi x_2}{H}\right)d\mathbf x,
 \]
 \[\mu:=\frac{8}{LH} \int_D(\tilde\omega^s)^2\cos  \left(\frac{2\pi x_2}{H}\right)d\mathbf x.\]

  To summarize, we have proved that $(a,b)$ and $(a',b')$ are two solutions to the following system of algebraic equations for some $\alpha,\beta\in\mathbb R$:
  \begin{equation}\label{abeq0} 
  \begin{cases} 2 x^2 +  y^2+px =\alpha,\\
  -xy^2+qx^2+ry^2+\mu x=\beta,\\
  y\geq 0.
  \end{cases}
  \end{equation}
Inserting the first equation of \eqref{abeq0} into the second one gives
  \begin{equation}\label{abeq1} 2x^3+(p+q-2r)x^2+(\mu-rp-\alpha)x=\beta-r\alpha.
  \end{equation} 
From \eqref{abeq1}, we can easily deduce the following two facts about \eqref{abeq0}:
\begin{itemize}
  \item [(a)] there are at most a finite number of solutions, and
  \item [(b)] if $(x,y)$ and $(x',y')$ are two different solutions, then  $x\neq x'.$
\end{itemize}
 Since we have assumed  that $v\notin\mathcal O_{\bar\omega},$ it follows that $(a,b)\neq (a',b')$, and thus $a\neq a'$ from the above discussion.   Then, by orthogonality, 
\[\|v-\bar\omega\|_{L^2(D)}> \delta_0\]
for some $\delta_0>0$ depending only on $\bar\omega.$ In combination with the fact that $\mathcal R_{\bar\omega}$ is bounded in $L^\infty(D)$, we further deduce that 
 \[\|v-\bar\omega\|_{L^p(D)}>\delta_1\]
for some $\delta_1>0$ depending only on $\bar\omega.$ This finishes the proof.

\end{proof}

\begin{remark}
In general, one cannot expect that $\mathcal M_{\bar\omega}=\mathcal O_{\bar\omega}.$ For example, when
\[\bar\omega=\cos\left(\frac{2\pi x_2}{H}\right)\in\mathbb E_1,\]
it holds that $\mathcal O_{\bar\omega}=\{\bar\omega\}$,
while $\mathcal M_{\bar\omega}=\{\bar\omega,\,-\bar\omega\}$.

\end{remark}

\begin{proof}[Proof of Theorem \ref{thm2}]
   The rigidity assertion (i) follows from Lemma \ref{lem422}, and the stability assertion (ii) follows from Theorem \ref{bsc} and Lemma \ref{lem423}.  
\end{proof}

\bigskip

 \noindent{\bf Acknowledgements:}
G. Wang was supported by NNSF of China grant  12471101.

\bigskip
\noindent{\bf  Data Availability} Data sharing not applicable to this article as no datasets were generated or analysed during the current study.

\bigskip
\noindent{\bf Declarations}

\bigskip
\noindent{\bf Conflict of interest}  The author declare that they have no conflict of interest to this work.

\phantom{s}
 \thispagestyle{empty}

\end{document}